\documentclass[10pt]{amsart} 
\usepackage{amsmath,amssymb,bm, color}

\usepackage{amsmath}
\usepackage{graphicx} 

\usepackage{mathrsfs}
\usepackage{bbm}
\usepackage{bm}
\usepackage{amsfonts,amssymb}
\usepackage{multirow}
\usepackage{lineno}
\usepackage{color}

\numberwithin{equation}{section}
 \newtheorem{theorem}{Theorem}[section]
 \newtheorem{lemma}[theorem]{Lemma}

\def\3bar{{|\hspace{-.02in}|\hspace{-.02in}|}}
\def\E{{\mathcal{E}}}
\def\T{{\mathcal{T}}}

\def\btau{\boldsymbol{\tau}}
\def\beta{\boldsymbol{\eta}}

\def\cal#1{{\mathcal #1}}
\def\pT{{\partial T}}

\def\bu{{\mathbf{u}}}
\def\bv{{\mathbf{v}}}
\def\bn{{\mathbf{n}}}
\def\be{{\mathbf{e}}}

\def\btau{{\boldsymbol{\tau}}}
\def\bzeta{{\boldsymbol{\zeta}}}

\newtheorem{remark}{Remark}[section]
\newtheorem{algorithm}{Weak Galerkin Algorithm}[section]

\numberwithin{equation}{section}

\def\3bar{{|\hspace{-.02in}|\hspace{-.02in}|}}

 \def\cal#1{\mathcal{#1}}

\begin{document}

\title []
 {Simplified Weak Galerkin Finite Element Methods  for Biharmonic Equations on Non-Convex Polytopal Meshes}

\author {Chunmei Wang}
\address{Department of Mathematics, University of Florida, Gainesville, FL 32611, USA. }
\email{chunmei.wang@ufl.edu}
\thanks{The research of Chunmei Wang was partially supported by National Science Foundation Grant DMS-2136380.}

\begin{abstract}
 This paper presents a simplified  weak Galerkin (WG) finite element method for solving biharmonic equations avoiding the use of traditional stabilizers. The proposed WG method supports both convex and non-convex polytopal elements in finite element partitions, utilizing bubble functions as a critical analytical tool. The simplified WG method is  symmetric  and positive definite. Optimal-order error estimates are established for WG approximations in both the discrete $H^2$ norm and the $L^2$ norm.
\end{abstract}


 
\maketitle

\section{Introduction}  
In this paper, we aim to develop a novel weak Galerkin (WG) finite element method for the biharmonic equation  that is applicable to non-convex polytopal meshes and eliminates the need for traditional stabilizers. To this aim, we consider the biharmonic equation with Dirichlet and Neumann boundary conditions. The goal is to find an unknown function $u$ satisfying
\begin{equation}\label{model}
 \begin{split}
  \Delta^2 u=&f, \qquad\qquad \text{in}\quad \Omega,\\
 u=&\xi,\qquad\qquad \text{on}\quad \partial\Omega,\\
\frac{\partial u}{\partial \bn}=&\nu,\qquad\qquad \text{on}\quad \partial\Omega, 
 \end{split}
 \end{equation}
where $\Omega\subset \mathbb R^d$ is an open bounded domain with a Lipschitz continuous boundary $\partial\Omega$. The domain $\Omega$ considered in this paper can be of any dimension $d\geq 2$. For the sake of clarity in presentation, we will focus on the case where $d=2$ throughout this paper. However, the analysis presented here can be readily extended to higher dimensions ($d\geq 3$) without significant   modifications.

The variational formulation of the model problem \eqref{model} is as follows: Find an unknown function $u\in H^2(\Omega)$ satisfying $u|_{\partial\Omega}=\xi$ and $\frac{\partial u}{\partial \bn}|_{\partial\Omega}=\nu$,  and the following equation 
\begin{equation}\label{weak}
\sum_{i,j=1}^2(\partial^2_{ij} u, \partial^2_{ij} v)=(f, v), \qquad \forall v\in H_0^2(\Omega),
 \end{equation} 
  where $\partial^2_{ij}$ denotes the second order partial derivative with respect to $x_i$ and $x_j$,  and $H_0^2(\Omega)=\{v\in H^2(\Omega): v|_{\partial\Omega}=0,   \nabla v |_{\partial\Omega}=0\}$.

The  WG  finite element method offers an innovative framework for the numerical solution of partial differential equations (PDEs). This approach approximates differential operators within a structure inspired by the theory of distributions, particularly for piecewise polynomial functions. Unlike traditional methods, WG reduces the regularity requirements on approximating functions through the use of carefully designed stabilizers. Extensive studies have highlighted the versatility and effectiveness of WG methods across a wide range of model PDEs, as demonstrated by numerous references \cite{wg1, wg2, wg3, wg4, wg5, wg6, wg7, wg8, wg9, wg10, wg11, wg12, wg13, wg14, wg15, wg16, wg17, wg18, wg19, wg20, wg21, itera, wy3655}, establishing WG as a powerful tool in scientific computing.
The defining feature of WG methods lies in their innovative use of weak derivatives and weak continuities to construct numerical schemes based on the weak forms of the underlying PDEs. This unique structure provides WG methods with exceptional flexibility, enabling them to address a wide variety of PDEs while ensuring both stability and accuracy in their numerical solutions.

Within the WG framework, a notable advancement is the Primal-Dual Weak Galerkin (PDWG) method. This approach effectively addresses challenges that traditional numerical methods often encounter \cite{pdwg1, pdwg2, pdwg3, pdwg4, pdwg5, pdwg6, pdwg7, pdwg8, pdwg9, pdwg10, pdwg11, pdwg12, pdwg13, pdwg14, pdwg15}. The PDWG method conceptualizes numerical solutions as the constrained minimization of functionals, where the constraints mirror the weak formulation of PDEs through the use of weak derivatives. This formulation leads to an Euler-Lagrange equation that incorporates both the primal variables and the dual variables (Lagrange multipliers), resulting in a robust numerical scheme.

This paper presents a simplified formulation of the WG finite element method, capable of handling both convex and non-convex elements in finite element partitions. A key innovation of our method is the elimination of stabilizers through the use of higher-degree polynomials for computing weak second-order partial derivatives. This design preserves the size and global sparsity of the stiffness matrix while substantially reducing the programming complexity associated with traditional stabilizer-dependent methods. The method leverages bubble functions as a critical analytical tool, representing a significant improvement over existing stabilizer-free WG methods \cite{ye}, which are limited to convex polytopal meshes. Our approach is versatile, accommodating arbitrary dimensions and polynomial degrees in the discretization process. In contrast, prior stabilizer-free WG methods \cite{ye} often require specific polynomial degree combinations and are restricted to 2D or 3D settings.
Theoretical analysis establishes optimal error estimates for the WG approximations in both the discrete
$H^2$ norm and an 
$L^2$ norm.

 This paper is organized as follows. Section 2 provides a brief review of the definition of the weak second order partial derivative and its discrete counterpart. In Section 3, we introduce an efficient WG  scheme that eliminates the need for stabilization terms. Section 4 establishes the existence and uniqueness of the solution. The error equation for the proposed WG scheme is derived in Section 5. Section 6 focuses on obtaining the error estimate for the numerical approximation in the discrete 
$H^2$
 norm, while Section 7 extends the analysis to derive the error estimate in the
$L^2$
 norm.  

Throughout this paper, we adopt standard notations. Let $D$ be any open, bounded domain with a Lipschitz continuous boundary in $\mathbb{R}^d$. The inner product, semi-norm, and norm in the Sobolev space  $H^s(D)$ for any integer $s\geq0$ are denoted by $(\cdot,\cdot)_{s,D}$, $|\cdot|_{s,D}$ and $\|\cdot\|_{s,D}$ respectively. For simplicity, when the domain $D$ is $\Omega$, the subscript $D$ is omitted from these notations. In the case $s=0$, the notations  
$(\cdot,\cdot)_{0,D}$, $|\cdot|_{0,D}$ and $\|\cdot\|_{0,D}$ are further  simplified as $(\cdot,\cdot)_D$, $|\cdot|_D$ and $\|\cdot\|_D$, respectively.

\section{Discrete Weak Second Order Partial Derivatives}\label{Section:Hessian}
This section provides a brief review of the definition of weak second order partial derivatives and their discrete counterparts, as introduced in \cite{wg21}.

Let $T$ be a polygonal element with boundary $\partial T$. A weak function on $T$  is represented  as  $v=\{v_0, v_b, \bv_g\}$, where $v_0\in L^2(T)$, $v_b\in L^{2}(\partial T)$ and $\bv_g\in [L^2(\partial T)]^2$. The first component, $v_0$,  denotes the value of $v$ within the interior of $T$, while the second component, $v_b$, represents the value of $v$  on the boundary of $T$. The third component  $\bv_g\in \mathbb R^2$  with components $v_{gi}$ ($i=1, 2$) approximates the gradient $\nabla v$ on the boundary $\partial T$. In general, $v_b$ and $\bv_g$ are treated as independent of the traces of $v_0$ and $\nabla v_0$, respectively.  
 
The space of all weak functions on $T$, denote by $W(T)$, is defined as  
 \begin{equation*}\label{2.1}
 W(T)=\{v=\{v_0,v_b, \bv_g\}: v_0\in L^2(T), v_b\in L^{2}(\partial
 T), \bv_g \in [L^2(\partial T)]^2\}.
\end{equation*}
 
 The weak second order partial derivative,  $\partial^2_{ij, w}$, is a linear
 operator mapping 
 $W(T)$ to the dual space of $H^2(T)$. For any
 $v\in W(T)$, $\partial^2_{ij, w} v$ is defined as a bounded linear functional on $H^2(T)$, given by:
 \begin{equation*}\label{2.3}
  (\partial^2_{ij, w}v, \varphi)_T=(v_0, \partial^2_{ji}\varphi)_T-
  \langle v_b n_i, \partial_j \varphi \rangle_{\partial T}+\langle v_{gi}, \varphi n_j\rangle_{\partial T},\quad \forall \varphi\in H^2(T),
  \end{equation*}
 where $ \bn$, with components $n_i (i=1,2)$,  represents the unit outward normal vector to $\partial T$.
 
 For any non-negative integer $r\ge 0$, let $P_r(T)$ denote the space of
 polynomials on $T$ with total degree at most 
 $r$. A discrete weak
second order partial derivative,  $\partial^2_{ij, w, r, T}$, is a linear operator
 mapping 
 $W(T)$ to $P_r(T)$. For any $v\in W(T)$,
 $\partial^2_{ij, w, r, T}v$ is  the unique polynomial  in $P_r(T)$ satisfying
 \begin{equation}\label{2.4}
 (\partial^2_{ij, w, r, T}v, \varphi)_T=(v_0, \partial^2_{ji}\varphi)_T-
  \langle v_b n_i, \partial_j \varphi \rangle_{\partial T}+\langle v_{gi}, \varphi n_j\rangle_{\partial T},\quad \forall \varphi \in P_r(T).
  \end{equation}
 For a smooth $v_0\in
 H^2(T)$, applying  standard  integration by parts to the first
 term on the right-hand side of (\ref{2.4})  yields:
 \begin{equation}\label{2.4new}
  (\partial^2_{ij, w, r, T}v, \varphi)_T=(\partial^2_{ij}v_0,  \varphi)_T-
  \langle (v_b-v_0) n_i, \partial_j \varphi \rangle_{\partial T}+\langle v_{gi}-\partial_i v_0, \varphi n_j\rangle_{\partial T},  
  \end{equation} 
 for any $ \varphi \in P_r(T)$.

\section{Weak Galerkin Algorithms without Stabilization Terms}\label{Section:WGFEM}
 Let ${\cal T}_h$ be a finite element partition of the domain
 $\Omega\subset \mathbb R^2$ into polygons. Assume that ${\cal
 T}_h$ satisfies  the shape regularity condition    \cite{wy3655}.
Let ${\mathcal E}_h$ represent the set of all edges in
 ${\cal T}_h$, and denote the set of interior edges by ${\mathcal E}_h^0={\mathcal E}_h \setminus
 \partial\Omega$. For any element $T\in {\cal T}_h$, let $h_T$ be its diameter, and define the mesh size as 
 $h=\max_{T\in {\cal
 T}_h}h_T$.

 Let $k$, $p$ and $q$ be integers such that $k\geq p\geq q\geq 1$. For any element $T\in\T_h$, the local weak finite element space is defined as:
 \begin{equation*}
 V(k, p, q, T)=\{\{v_0,v_b, \bv_g\}: v_0\in P_k(T), v_b\in P_{p}(e), \bv_g\in [P_{q}(e)]^2, e\subset \partial T\}.    
 \end{equation*}
By combining the local spaces $V(k, p, q,T)$ across all elements $T\in {\cal T}_h$ and ensuring continuity of  $v_b$ and $\bv_g$   along the interior edges  $\E_h^0$,  we obtain the global weak finite element space:
 $$
 V_h=\big\{\{v_0,v_b, \bv_g\}:\ \{v_0,v_b, \bv_g\}|_T\in V(k, p, q,  T),
 \forall T\in {\cal T}_h \big\}.
 $$ 
The subspace of $V_h$ consisting of functions with vanishing boundary values on $\partial\Omega$ is defined as:
$$
V_h^0=\{v\in V_h: v_b|_{\partial\Omega}=0, \bv_g|_{\partial\Omega}=0\}.
$$

For simplicity, the discrete weak second order partial derivative  $\partial^2_{ij, w}v$ is used to denote the operator 
$\partial^2_{ij, w, r, T} v$ defined in 
(\ref{2.4}) on each element $T\in {\cal T}_h$,  as:
$$
(\partial^2_{ij, w} v)|_T= \partial^2_{ij, w, r, T}(v |_T), \qquad \forall T\in \T_h.
$$
 
On each element $T\in\T_h$, let $Q_0$  denote the $L^2$ projection onto $P_k(T)$. On each edge $e\subset\partial T$, let $Q_b$ and $Q_n$  denote the $L^2$ projection operators onto $P_{p}(e)$ and $P_{q}(e)$, respectively. 
 For any $w\in H^2(\Omega)$,  the $L^2$ projection  into the weak finite element space $V_h$ is denoted by $Q_h w$, defined as:
 $$
  (Q_hw)|_T:=\{Q_0(w|_T),Q_b(w|_{\pT}), Q_n(\nabla w|_{\pT})\},\qquad \forall T\in\T_h.
$$

The  simplified WG numerical scheme, free from stabilization terms, for solving the biharmonic equation \eqref{model} is formulated as follows:
\begin{algorithm}\label{PDWG1}
Find $u_h=\{u_0, u_b, \bu_g\} \in V_h$ such that $u_b=Q_b\xi$, $\bu_g\cdot\bn=Q_n\nu$ and $\bu_g\cdot\btau=Q_n(\nabla\xi\cdot\btau)$ on $\partial\Omega$, and satisfy:
\begin{equation}\label{WG}
(\partial^2_{w} u_h, \partial^2_{w} v)=(f, v_0), \qquad\forall v=\{v_0, v_b, \bv_g\}\in V_h^0,
\end{equation}
where $\btau\in \mathbb R^2$ is the tangential direction along $\partial\Omega$, and the terms are defined as:
$$
(\partial^2_{w} u_h, \partial^2_{w} v)=\sum_{T\in {\cal T}_h}\sum_{i,j=1}^2 (\partial^2_{ij, w} u_h, \partial^2_{ij, w} v)_T,
$$ 
$$
(f, v_0)=\sum_{T\in {\cal T}_h}(f, v_0)_T.
$$
\end{algorithm}

\section{Solution Existence and Uniqueness}  
Recall that ${\cal T}_h$ is a shape-regular finite element partition of the domain $\Omega$. Consequently, for any  $T\in {\cal T}_h$ and $\phi\in H^1(T)$, the following trace inequality holds \cite{wy3655}: 
\begin{equation}\label{tracein}
 \|\phi\|^2_{\partial T} \leq C(h_T^{-1}\|\phi\|_T^2+h_T \|\nabla \phi\|_T^2).
\end{equation}
If $\phi$ is a polynomial on the element $T\in {\cal T}_h$,  a simpler form of the trace inequality applies \cite{wy3655}: 
\begin{equation}\label{trace}
\|\phi\|^2_{\partial T} \leq Ch_T^{-1}\|\phi\|_T^2.
\end{equation}

For any $v=\{v_0, v_b, \bv_g\}\in V_h$, define the norm: 
\begin{equation}\label{3norm}
\3bar v\3bar= (\partial^2_{w} v, \partial^2_{ w} v) ^{\frac{1}{2}},
\end{equation}
and introduce the discrete  $H^2$- semi-norm: 
\begin{equation}\label{disnorm}
\|v\|_{2, h}=\Big( \sum_{T\in {\cal T}_h} \|\sum_{i,j=1}^2\partial^2_{ij} v_0\|_T^2+h_T^{-3}\|v_0-v_b\|_{\partial T}^2+h_T^{-1}\|\nabla v_0-\bv_g\|_{\partial T}^2\Big)^{\frac{1}{2}}.
\end{equation}
\begin{lemma}\label{norm1}
 For $v=\{v_0, v_b, \bv_g\}\in V_h$, there exists a constant $C$ such that for $i, j=1, 2$,
 $$
 \|\partial^2_{ij} v_0\|_T\leq C\|\partial^2_{ij, w} v\|_T.
 $$
\end{lemma}
\begin{proof} Let $T\in {\cal T}_h$ be a polytopal element with $N$ edges denoted as  $e_1, \cdots, e_N$. Importantly, $T$ can be non-convex. On each edge $e_i$,  construct a linear function  $l_i(x)$ satisfying  $l_i(x)=0$ on $e_i$ as: 
$$
l_i(x)=\frac{1}{h_T}\overrightarrow{AX}\cdot \bn_i, $$  where  $A$ is a fixed point on $e_i$,  $X$ is any point on $e_i$, $\bn_i$ is the normal vector to  $e_i$, and $h_T$ is the diameter of $T$. 

Define the bubble function for  $T$ as:
 $$
 \Phi_B =l^2_1(x)l^2_2(x)\cdots l^2_N(x) \in P_{2N}(T).
 $$ 
 It is straightforward to verify that $\Phi_B=0$ on  $\partial T$.    The function 
  $\Phi_B$  can be scaled such that $\Phi_B(M)=1$ where   $M$ is the barycenter of  $T$. Additionally, there exists a subdomain $\hat{T}\subset T$ such that $\Phi_B\geq \rho_0$ for some constant $\rho_0>0$.
  
For $v=\{v_0, v_b, \bv_g\}\in V_h$, let $r=2N+k-2$ and choose   $\varphi=\Phi_B \partial^2_{ij} v_0\in P_r(T)$ in \eqref{2.4new}. This yields:
\begin{equation}\label{t1}
\begin{split}
     &(\partial^2_{ij, w}v, \Phi_B \partial^2_{ij} v_0)_T\\=&(\partial^2_{ij}v_0,  \Phi_B  \partial^2_{ij} v_0)_T-
  \langle (v_b-v_0) n_i, \partial_j (\Phi_B  \partial^2_{ij} v_0 )\rangle_{\partial T} \\&+\langle v_{gi}-\partial_i v_0, \Phi_B  \partial^2_{ij} v_0 n_j\rangle_{\partial T}\\
  =& (\partial^2_{ij}v_0,  \Phi_B \partial^2_{ij} v_0)_T,  
\end{split}
  \end{equation}
where  we applied $\Phi_B=0$  on $\partial T$.

Using the domain inverse inequality \cite{wy3655}, there exists a constant $C$ such that 
\begin{equation}\label{t2}
(\partial^2_{ij} v_0, \Phi_B \partial^2_{ij} v_0)_T \geq C (\partial^2_{ij} v_0, \partial^2_{ij} v_0)_T.
\end{equation}

By applying the Cauchy-Schwarz inequality to   \eqref{t1}-\eqref{t2}, we obtain
 $$
 (\partial^2_{ij} v_0, \partial^2_{ij} v_0)_T\leq C (\partial^2_{ij, w} v, \Phi_B \partial^2_{ij} v_0)_T  \leq C  \|\partial^2_{ij, w} v\|_T \|\Phi_B \partial^2_{ij} v_0\|_T  \leq C
\|\partial^2_{ij, w} v\|_T \|\partial^2_{ij} v_0\|_T,
 $$
which implies:
 $$
 \|\partial^2_{ij}v_0\|_T\leq C\|\partial^2_{ij, w} v\|_T.
 $$

This completes the proof.
\end{proof}

\begin{remark}
   If the polytopal element $T$  is convex, 
   the bubble function  in Lemma \ref{norm1}  can be  simplified to:
 $$
 \Phi_B =l_1(x)l_2(x)\cdots l_N(x).
 $$  
This simplified bubble function satisfies (1) $\Phi_B=0$
 on $\partial T$, (2) there exists a subdomain $\hat{T}\subset T$ such that $\Phi_B\geq \rho_0$ for some constant $\rho_0>0$.
The proof of Lemma \ref{norm1} follows the same approach, using this simplified bubble function. In this case, we set $r=N+k-2$.  
\end{remark}

By constructing an  edge-based bubble function,    $$\varphi_{e_k}= \Pi_{i=1, \cdots, N, i\neq k}l_i^2(x),$$   it can be easily verified that (1) $\varphi_{e_k}=0$ on each  edge  $e_i$ for $i \neq k$, and (2) there exists a subdomain $\widehat{e_k}\subset e_k$ such that  $\varphi_{e_k} \geq \rho_1$ for some constant $\rho_1>0$. Let $\varphi=(v_b-v_0)l_k \varphi_{e_k}$. It is straightforward to verify the following properties: (1) $\varphi=0$ on each edge  $e_i$ for $i=1, \cdots, N$,  (2) $\nabla \varphi =0$ on each edge   $e_i$ for $i \neq k$,  and (3) $\nabla \varphi =(v_0-v_b)(\nabla l_k) \varphi_{e_k}=\mathcal{O}( \frac{ (v_0-v_b)\varphi_{e_k}}{h_T}\textbf{C})$  on  $e_k$, where  $\textbf{C}$ is a constant vector.

\begin{lemma}\cite{autobihar}\label{phi}
     For $\{v_0,v_b, \bv_g\}\in V_h$, let $\varphi=(v_b-v_0)l_k \varphi_{e_k}$. The following inequality holds:
\begin{equation}
  \|\varphi\|_T ^2 \leq Ch_T  \int_{e_k}(v_b-v_0)^2ds.
\end{equation}
\end{lemma}

\begin{lemma}\label{phi2}
     For $\{v_0,v_b, \bv_g\}\in V_h$, let $\varphi=(v_{gi} -\partial_i v_0 )  \varphi_{e_k}$. The following inequality holds:
\begin{equation}
  \|\varphi\|_T ^2 \leq Ch_T \int_{e_k}(v_{gi} -\partial_i v_0 )^2ds.
\end{equation}
\end{lemma}
\begin{proof}
Define the extension of $\bv_g$, originally defined on the edge  $e_k$, to the entire  polytopal element $T$ as: 
$$
\bv_g(X)= \bv_g(Proj_{e_k} (X)),
$$
where $X=(x_1, x_2)$ is any point in   $T$, $Proj_{e_k} (X)$ denotes the orthogonal projection of   $X$ onto the  plane $H\subset\mathbb R^2$ containing     $e_k$.  If  $Proj_{e_k} (X)$ is not on  $e_k$, $\bv_g(Proj_{e_k} (X))$ is defined as the extension of $\bv_g$ from $e_k$ to   $H$. The extension preserves the polynomial nature of $\bv_g$  as demonstrated in   \cite{autobihar}.

Let $v_{trace}$ denote the trace of $v_0$  on   $e_k$. Define its extension to  $T$ as:
$$
 v_{trace} (X)= v_{trace}(Proj_{e_k} (X)).
$$
This extension is also polynomial, as demonstrated in   \cite{autobihar}.

Let $\varphi=(v_{gi} -\partial_i v_0 )  \varphi_{e_k}$. Then,
\begin{equation*}
    \begin{split}
\|\varphi\|^2_T  =
\int_T \varphi^2dT = 
&\int_T ((v_{gi} -\partial_i v_0 )(X)  \varphi_{e_k})^2dT\\
\leq &Ch_T \int_{e_k} ((v_{gi} -\partial_i v_0 )(Proj_{e_k} (X)) \varphi_{e_k})^2dT\\ 
\\\leq &Ch_T \int_{e_k}  (v_{gi} -\partial_i v_0 ) ^2ds, 
    \end{split}
\end{equation*} 
where we used the facts that (1) $\varphi_{e_k}=0$ on each  edge   $e_i$ for $i \neq k$,  (2) there exists a subdomain $\widehat{e_k}\subset e_k$ such that  $\varphi_{e_k} \geq \rho_1$ for some constant $\rho_1>0$, 
and applied the properties of the projection.

 This completes the proof of the lemma.

\end{proof}

\begin{lemma}\label{normeqva}   There exist  two positive constants, $C_1$ and $C_2$, such that for any $v=\{v_0, v_b, \bv_g\} \in V_h$,   the following equivalence holds: 
 \begin{equation}\label{normeq}
 C_1\|v\|_{2, h}\leq \3bar v\3bar  \leq C_2\|v\|_{2, h}.
\end{equation}
\end{lemma} 

\begin{proof}   
Consider the edge-based bubble function defined as 
$$\varphi_{e_k}= \Pi_{i=1, \cdots, N, i\neq k}l_i^2(x).$$

First, extend $v_b$ from the edge  $e_k$ to the element $T$. 
Similarly, let  $v_{trace}$ denote the trace of $v_0$ on the edge  $e_k$ and extend $v_{trace}$ to the element $T$. For simplicity, we continue to use $v_b$ and $v_0$ to represent their respective extensions. Details of these extensions can be found in Lemma \ref{phi2}. Substituting $\varphi=(v_b-v_0)l_k\varphi_{e_k}$ into \eqref{2.4new}, we obtain
  \begin{equation} \label{t33}
 \begin{split} 
    (\partial^2_{ij, w}v, \varphi)_T =& (\partial^2_{ij}v_0,  \varphi)_T-
  \langle (v_b-v_0) n_i, \partial_j \varphi \rangle_{\partial T}+\langle v_{gi}-\partial_i v_0, \varphi n_j\rangle_{\partial T}\\
  =& (\partial^2_{ij}v_0,  \varphi)_T +  Ch_T^{-1}\int_{e_k} |v_b-  v_0|^2 \varphi_{e_k} ds,
  \end{split}
  \end{equation}  
where we used (1) $\varphi=0$ on each edge  $e_i$ for $i=1, \cdots, N$,  (2) $\nabla \varphi =0$ on each edge   $e_i$ for $i \neq k$,  and (3) $\nabla \varphi =(v_0-v_b)(\nabla l_k) \varphi_{e_k}=\mathcal{O}( \frac{ (v_0-v_b)\varphi_{e_k}}{h_T}\textbf{C})$  on  $e_k$, where  $\textbf{C}$ is a constant vector.
  
Recall that  (1) $\varphi_{e_k}=0$ on each  edge   $e_i$ for $i \neq k$, and (2) there exists a subdomain $\widehat{e_k}\subset e_k$ such that  $\varphi_{e_k} \geq \rho_1$ for some constant $\rho_1>0$.
Using Cauchy-Schwarz inequality, the domain inverse inequality \cite{wy3655},  \eqref{t33} and Lemma \ref{phi}, we deduce:
\begin{equation*}
\begin{split}
  \int_{e_k}|v_b-  v_0|^2  ds\leq & C\int_{e_k} |v_b-  v_0|^2 \varphi_{e_k} ds 
  \\ \leq& C h_T(\|\partial^2_{ij, w} v\|_T+\|\partial^2_{ij} v_0\|_T){ \| \varphi\|_T}\\
 \leq & C { h_T^{\frac{3}{2}}} (\|\partial^2_{ij, w} v\|_T+\|\partial^2_{ij} v_0\|_T){ (\int_{e_k}|v_b- v_0|^2ds)^{\frac{1}{2}}},
 \end{split}
\end{equation*}
which, from Lemma \ref{norm1}, leads to:
\begin{equation}\label{t21}
 h_T^{-3}\int_{e_k}|v_b-  v_0|^2  ds \leq C  (\|\partial^2_{ij, w} v\|^2_T+\|\partial^2_{ij} v_0\|^2_T)\leq C\|\partial^2_{ij, w} v\|^2_T.   
\end{equation} 
Next,  extend $\bv_g$ from the edge  $e_k$ to the element $T$, denoting the extension by the same symbol for simplicity.   Details of this extension are in Lemma \ref{phi2}.
Substituting $\varphi=(v_{gi}-\partial_i v_0)\varphi_{e_k}$ into \eqref{2.4new},  we obtain: 
  \begin{equation} \label{t3}
 \begin{split} & (\partial^2_{ij, w}v, \varphi)_T\\=& (\partial^2_{ij}v_0,  \varphi)_T-
  \langle (v_b-v_0) n_i, \partial_j \varphi \rangle_{\partial T}+\langle v_{gi}-\partial_i v_0, \varphi n_j\rangle_{\partial T}\\
  =& (\partial^2_{ij}v_0,  \varphi)_T -
  \langle (v_b-v_0) n_i, \partial_j \varphi \rangle_{\partial T}+\int_{e_k} |v_{gi}-\partial_i v_0|^2  \varphi_{e_k}ds,
  \end{split}
  \end{equation}  
  where we used  $\varphi_{e_k} =0$ on edge  $e_i$ for $i \neq k$, 
and the fact that there exists a sub-domain $\widehat{e_k}\subset e_k$ such that  $\varphi_{e_k} \geq \rho_1$ for some constant $\rho_1>0$.
This, together with   Cauchy-Schwarz inequality, the domain inverse inequality \cite{wy3655},  the inverse inequality, the trace inequality \eqref{trace}, \eqref{t21} and Lemma \ref{phi2}, gives
 \begin{equation*} 
\begin{split}
&  \int_{e_k}|v_{gi}-\partial_i v_0|^2  ds\\\leq &C  \int_{e_k}|v_{gi}-\partial_i v_0|^2  \varphi_{e_k}ds\\
  \leq & C (\|\partial^2_{ij, w} v\|_T+\|\partial^2_{ij} v_0\|_T)\| \varphi\|_T+  C\|v_0-v_b\|_{\partial T}\|\partial_j \phi\|_{\partial T}\\
 \leq & C h_T^{\frac{1}{2}} (\|\partial^2_{ij, w} v\|_T+\|\partial^2_{ij} v_0\|_T)(\int_{e_k}|v_{gi}-\partial_i v_0|^2ds)^{\frac{1}{2}}\\&+ C h_T^{\frac{3}{2}}  \|\partial^2_{ij, w} v\|_T  h_T^{-1}(\int_{e_k}|v_{gi}-\partial_i v_0|^2ds)^{\frac{1}{2}}. 
 \end{split}
\end{equation*} 
Applying Lemma \ref{norm1}, gives 
\begin{equation} \label{t11}
 h_T^{-1}\int_{e_k}|v_{gi}-\partial_i v_0|^2  ds \leq C  (\|\partial^2_{ij, w} v\|^2_T+\|\partial^2_{ij} v_0\|^2_T)\leq C\|\partial^2_{ij, w} v\|^2_T.
\end{equation}
  Using Lemma \ref{norm1},  equations  \eqref{t21}, and \eqref{t11}, \eqref{3norm} and \eqref{disnorm},  we deduce:
$$
 C_1\|v\|_{2, h}\leq \3bar v\3bar.
$$

Finally, using the Cauchy-Schwarz inequality, inverse inequalities, and the trace inequality \eqref{trace} in \eqref{2.4new}, we derive: 
\begin{equation*}
    \begin{split}
  \Big| (\partial^2_{ij, w}v, \varphi)_T\Big| \leq &\|\partial^2_{ij}v_0\|_T \|  \varphi\|_T+
 \|(v_b-v_0) n_i\|_{\partial T} \| \partial_j\varphi\|_{\partial T}+\|v_{gi}-\partial_i v_0\|_{\partial T} \|\varphi n_j\|_{\partial T} \\
 \leq &\|\partial^2_{ij}v_0\|_T \|  \varphi\|_T+
 h_T^{-\frac{3}{2}}\|v_b-v_0\|_{\partial T} \|  \varphi\|_{  T}+h_T^{-\frac{1}{2}}\|v_{gi}-\partial_i v_0\|_{\partial T} \|\varphi  \|_{T},
    \end{split}
\end{equation*}
which gives:
$$
\| \partial^2_{ij, w}v\|_T^2\leq C( \|\partial^2_{ij}v_0\|^2_T  +
 h_T^{-3}\|v_b-v_0\|^2_{\partial T}+h_T^{-1}\|v_{gi}-\partial_i v_0\|^2_{\partial T}),
$$
 and further gives $$ \3bar v\3bar  \leq C_2\|v\|_{2, h}.$$

 This completes the proof. 
 \end{proof}
  
\begin{theorem}
The  WG scheme \ref{PDWG1} admits  a unique solution. 
\end{theorem}
\begin{proof}
Assume that $u_h^{(1)}\in V_h$ and $u_h^{(2)}\in V_h$ are two  distinct solutions of the WG scheme \ref{PDWG1}. Define  $\eta_h= u_h^{(1)}-u_h^{(2)}\in V_h^0$. Then, $\eta_h$ satisfies 
$$
( \partial^2_{ij, w} \eta_h,  \partial^2_{ij, w} v)=0, \qquad \forall v\in V_h^0.
$$
Choosing $v=\eta_h$  yields $\3bar \eta_h\3bar=0$. From the equivalence of norms in  \eqref{normeq},   it follows that $\|\eta_h\|_{2,h}=0$, which yields
$\partial^2_{ij} \eta_0=0$ for $i,j =1, 2$ on each $T$,  $\eta_0=\eta_b$ and $\nabla \eta_0=\beta_g$ on each $\partial T$. Consequently,  $\eta_0$ is a linear function on each  element $T$ and $\nabla \eta_0=C$  on each  $T$.

Since $\nabla \eta_0=\beta_g$ on each $\partial T$,  it follows that  $\nabla \eta_0$ is continuous across the entire domain $\Omega$. Thus,  $\nabla \eta_0=C$ throughout $\Omega$.  Furthermore, the condition $\beta_g=0$ on $\partial\Omega$ implies $\nabla \eta_0=0$ in $\Omega$ and $\beta_g=0$ on each $\partial T$. Therefore, $\eta_0$ is a constant on each element $T$. 

Since $\eta_0=\eta_b$ on $\partial T$,  the continuity of  $\eta_0$  over $\Omega$ implies $\eta_0$ is globally constant. From   $\eta_b=0$ on $\partial\Omega$, we conclude $\eta_0=0$ throughout $\Omega$. Consequently, $\eta_b=\eta_0=0$ on each $\partial T$, which implies $\eta_h\equiv 0$ in  $\Omega$. Thus, 
 $u_h^{(1)}\equiv u_h^{(2)}$, proving the uniqueness of the solution.
\end{proof}

\section{Error Equations}
Let $Q_r$ denote the $L^2$ projection operator onto the finite element space of piecewise polynomials of degree at most $r$.

\begin{lemma}\label{Lemma5.1}   The following property holds: 
\begin{equation}\label{pro}
\partial^2_{ij, w}u =Q_r(\partial^2_{ij} u), \qquad \forall u\in H^2(T).
\end{equation}
\end{lemma}

\begin{proof} For any $u\in H^2(T)$,  using \eqref{2.4new}, we have  
 \begin{equation*} 
  \begin{split}
 &(\partial^2_{ij, w}u, \varphi)_T\\
  =&(\partial^2_{ij}u,  \varphi)_T-
  \langle (u|_{\partial T}-u|_T) n_i, \partial_j \varphi \rangle_{\partial T}+\langle (\nabla u|_{\partial T})_{i} -\partial_i (u|_{T}), \varphi n_j\rangle_{\partial T}\\
  =&(\partial^2_{ij}u,  \varphi)_T=(Q_r(\partial^2_{ij}u),  \varphi)_T,
   \end{split}
   \end{equation*} 
  for all $\varphi\in P_r(T)$.  This  completes the proof.
  \end{proof}

Let  $u$ be the exact solution of the biharmonic equation \eqref{model}, and 
  $u_h \in V_{h}$  its numerical approximation obtained from the WG scheme \ref{PDWG1}. The error function, denoted by  $e_h$, is defined as 
\begin{equation}\label{error} 
e_h=u-u_h.
\end{equation}

\begin{lemma}\label{errorequa}
The error function $e_h$ defined in (\ref{error}) satisfies the following error equation:
\begin{equation}\label{erroreqn}
(\partial_{w}^2 e_h, \partial_{w}^2 v)=\ell (u, v), \qquad \forall v\in V_h^0,
\end{equation}
where 
$$
\ell (u, v)=\sum_{T\in {\cal T}_h}\sum_{i,j=1}^2 -  \langle (v_b-v_0) n_i, \partial_j ((Q_r-I) \partial_{ij}^2 u) \rangle_{\partial T}+\langle v_{gi}-\partial_i v_0, (Q_r-I) \partial_{ij}^2 u n_j\rangle_{\partial T}.
$$
\end{lemma}
\begin{proof}   Using \eqref{pro}, standard  integration by parts, and substituting  $\varphi= Q_r \partial_{ij}^2 u$ into \eqref{2.4new}, we obtain 
\begin{equation}\label{54}
\begin{split}
&\sum_{T\in {\cal T}_h}\sum_{i,j=1}^2(\partial_{ij, w}^2 u, \partial_{ij, w}^2 v)_T\\=&\sum_{T\in {\cal T}_h}\sum_{i,j=1}^2(Q_r \partial_{ij}^2 u, \partial_{ij, w}^2 v)_T\\ 
=&\sum_{T\in {\cal T}_h}\sum_{i,j=1}^2(\partial^2_{ij}v_0,  Q_r \partial_{ij}^2 u)_T-
  \langle (v_b-v_0) n_i, \partial_j (Q_r \partial_{ij}^2 u) \rangle_{\partial T}\\&+\langle v_{gi}-\partial_i v_0, Q_r \partial_{ij}^2 u n_j\rangle_{\partial T}\\
=& \sum_{T\in {\cal T}_h}\sum_{i,j=1}^2(\partial^2_{ij}v_0,   \partial_{ij}^2 u)_T-
  \langle (v_b-v_0) n_i, \partial_j (Q_r \partial_{ij}^2 u) \rangle_{\partial T}\\&+\langle v_{gi}-\partial_i v_0, Q_r \partial_{ij}^2 u n_j\rangle_{\partial T}\\
=& \sum_{T\in {\cal T}_h}\sum_{i,j=1}^2 ((\partial^2_{ij})^2u, v_0)_T+\langle \partial_{ij}^2 u, \partial_i v_0\cdot n_j\rangle_{\partial T}-\langle \partial_j(\partial_{ij}^2u)\cdot n_i, v_0\rangle_{\partial T}\\
&-  \langle (v_b-v_0) n_i, \partial_j (Q_r \partial_{ij}^2 u) \rangle_{\partial T}+\langle v_{gi}-\partial_i v_0, Q_r \partial_{ij}^2 u n_j\rangle_{\partial T}\\
  =&(f, v_0)+\sum_{T\in {\cal T}_h}\sum_{i,j=1}^2 -  \langle (v_b-v_0) n_i, \partial_j ((Q_r-I) \partial_{ij}^2 u) \rangle_{\partial T}\\&+\langle v_{gi}-\partial_i v_0, (Q_r-I) \partial_{ij}^2 u n_j\rangle_{\partial T},
\end{split}
\end{equation}
where we used \eqref{model}, $\partial_{ij}^2 v_0\in P_{k-2}(T)$ and $r=2N+k-2\geq k-2$,  $\sum_{T\in {\cal T}_h} \sum_{i,j=1}^2  \langle \partial_{ij}^2 u,  v_{gi}\cdot n_j\rangle_{\partial T}=\sum_{T\in {\cal T}_h} \sum_{i,j=1}^2  \langle \partial_{ij}^2 u,  v_{gi}\cdot n_j\rangle_{\partial \Omega}=0$ since $v_{gi}=0$ on $\partial \Omega$, and 
$\sum_{T\in {\cal T}_h} \sum_{i,j=1}^2 \langle \partial_j(\partial_{ij}^2u)\cdot n_i, v_b\rangle_{\partial T}= \sum_{T\in {\cal T}_h} \sum_{i,j=1}^2 \langle \partial_j(\partial_{ij}^2u)\cdot n_i, v_b\rangle_{\partial \Omega}=0$ since $v_{b}=0$ on $\partial \Omega$.  

Subtracting \eqref{WG} from \eqref{54}  yields  
\begin{equation*}  
\begin{split}
&\sum_{T\in {\cal T}_h}\sum_{i,j=1}^2(\partial_{ij, w}^2 e_h, \partial_{ij, w}^2 v)_T\\=&\sum_{T\in {\cal T}_h}\sum_{i,j=1}^2 -  \langle (v_b-v_0) n_i, \partial_j ((Q_r-I) \partial_{ij}^2 u) \rangle_{\partial T}+\langle v_{gi}-\partial_i v_0, (Q_r-I) \partial_{ij}^2 u n_j\rangle_{\partial T}.
\end{split}
\end{equation*}
This concludes the proof.
\end{proof}

\section{Error Estimates in the $H^2$ Norm}

\begin{lemma}\cite{wg21}\label{lem}
Let ${\cal T}_h$ be a finite element partition of the domain $\Omega$ satisfying the shape  regularity  assumption  specified in \cite{wy3655}. For any $0\leq s \leq 2$ and $1\leq m \leq k$, the following estimates hold:
\begin{eqnarray}\label{error1}
 \sum_{T\in {\cal T}_h}\sum_{i,j=1}^2 h_T^{2s}\|\partial_{ij}^2 u- Q_r \partial_{ij}^2 u\|^2_{s,T}&\leq& C  h^{2(m-1)}\|u\|^2_{m+1},\\
\label{error2}
\sum_{T\in {\cal T}_h}h_T^{2s}\|u- Q _0u\|^2_{s,T}&\leq& C h^{2(m+1)}\|u\|^2_{m+1}.
\end{eqnarray}
 \end{lemma}
 \begin{lemma}
If   $u\in H^{k+1}(\Omega)$, then there exists a constant 
$C$
 such that 
\begin{equation}\label{erroresti1}
\3bar u-Q_hu \3bar \leq Ch^{k-1}\|u\|_{k+1}.
\end{equation}
\end{lemma}
\begin{proof}
Utilizing \eqref{2.4new}, the trace inequalities \eqref{tracein} and \eqref{trace}, the inverse inequality, and the estimate \eqref{error2} for  $m=k$ and $s=0, 1, 2$, we analyze the following summation for any $\varphi\in P_r(T)$:
\begin{equation*}
\begin{split}
&\sum_{T\in {\cal T}_h}\sum_{i, j=1}^2(\partial^2_{ij, w}(u-Q_hu),  \varphi)_T\\
 = &\sum_{T\in {\cal T}_h}\sum_{i, j=1}^2 (\partial^2_{ij}(u-Q_0u),  \varphi)_T-
  \langle (Q_0u-Q_bu) n_i, \partial_j \varphi \rangle_{\partial T}\\&+\langle (\partial_i u- Q_n  (\partial_i u))-\partial_i (u-Q_0u), \varphi n_j\rangle_{\partial T}\\
\leq &\Big(\sum_{T\in {\cal T}_h}\sum_{i, j=1}^2\|\partial^2_{ij}(u-Q_0u)\|^2_T\Big)^{\frac{1}{2}} \Big(\sum_{T\in {\cal T}_h} \|\varphi\|_T^2\Big)^{\frac{1}{2}}\\&
 + \Big(\sum_{T\in {\cal T}_h} \sum_{i=1}^2\|(Q_0u-Q_bu) n_i\|_{\partial T} ^2\Big)^{\frac{1}{2}}\Big(\sum_{T\in {\cal T}_h} \sum_{ j=1}^2\|\partial_j \varphi\|_{\partial T}^2\Big)^{\frac{1}{2}}\\
 &+ \Big(\sum_{T\in {\cal T}_h} \sum_{i =1}^2\| \partial_i (Q_0u)- Q_n  (\partial_i u) \|_{\partial T} ^2\Big)^{\frac{1}{2}}\Big(\sum_{T\in {\cal T}_h}\sum_{ j=1}^2 \| \varphi n_j\|_{\partial T}^2\Big)^{\frac{1}{2}}\\
\leq &\Big(\sum_{T\in {\cal T}_h}\sum_{i, j=1}^2\|\partial^2_{ij}(u-Q_0u)\|^2_T\Big)^{\frac{1}{2}} \Big(\sum_{T\in {\cal T}_h} \|\varphi\|_T^2\Big)^{\frac{1}{2}}\\&
 + \Big(\sum_{T\in {\cal T}_h}  h_T^{-1}\| Q_0u- u \|_{  T}+h_T \| Q_0u- u \|_{1,  T} ^2\Big)^{\frac{1}{2}}\Big(\sum_{T\in {\cal T}_h}  h_T^{-3}\| \varphi\|_{ T}^2\Big)^{\frac{1}{2}}\\
 &+ \Big(\sum_{T\in {\cal T}_h} \sum_{i =1}^2h_T^{-1}\|  \partial_i (Q_0u)-  \partial_i u   \|_{ T} ^2+h_T \| \partial_i (Q_0u)-  \partial_i u \|_{1, T} ^2\Big)^{\frac{1}{2}}\Big(\sum_{T\in {\cal T}_h} h_T^{-1} \| \varphi \|_{T}^2\Big)^{\frac{1}{2}}\\
&\leq Ch^{k-1}\|u\|_{k+1}\Big(\sum_{T\in {\cal T}_h} \|\varphi\|_T^2\Big)^{\frac{1}{2}}.
\end{split}
\end{equation*}
Letting $\varphi=\partial^2_{ij, w}(u-Q_hu)$ gives 
$$
\sum_{T\in {\cal T}_h}\sum_{i,j=1}^2(\partial^2_{ij, w}(u-Q_hu), \partial^2_{ij, w}(u-Q_hu))_T\leq 
 Ch^{k-1}\|u\|_{k+1}\3bar u-Q_hu \3bar.$$  
 
 This completes the proof.
\end{proof}

\begin{theorem}
Suppose  the exact solution 
$u$
 of the biharmonic equation \eqref{model} satisfies 
$u\in H^{k+1}(\Omega)$. Then,  the error estimate satisfies:
\begin{equation}\label{trinorm}
\3bar u-u_h\3bar \leq Ch^{k-1}\|u\|_{k+1}.
\end{equation}
\end{theorem}
\begin{proof}
Note that $r\geq 1$. For the first term on the right-hand side of the error equation \eqref{erroreqn}, using Cauchy-Schwarz inequality, the trace inequality \eqref{tracein}, the estimate \eqref{error1} with $m=k$ and $s=1, 2$, and \eqref{normeq}, we have 
 \begin{equation}\label{erroreqn1}
\begin{split}
&\Big|\sum_{T\in {\cal T}_h}\sum_{i,j=1}^2 -  \langle (v_b-v_0) n_i, \partial_j ((Q_r-I) \partial_{ij}^2 u) \rangle_{\partial T}\Big|\\
\leq & C(\sum_{T\in {\cal T}_h}\sum_{i=1}^2h_T^{-3}\|(v_b-v_0) n_i\|^2_{\partial T} )^{\frac{1}{2}} \cdot(\sum_{T\in {\cal T}_h} \sum_{i,j=1}^2h_T^3\|\partial_j ((Q_r-I) \partial_{ij}^2 u) \|^2_{\partial T})^{\frac{1}{2}}\\\leq & C \| v\|_{2,h} (\sum_{T\in {\cal T}_h} \sum_{i,j=1}^2h_T^2\|\partial_j ((Q_r-I) \partial_{ij}^2 u) \|^2_{T}+h_T^4\|\partial_j ((Q_r-I) \partial_{ij}^2 u) \|^2_{1, T})^{\frac{1}{2}}\\
\leq & Ch^{k-1} \|u\|_{k+1}  \3bar v\3bar.
\end{split}
\end{equation}

For the second term on the right-hand side of the error equation \eqref{erroreqn}, using the Cauchy-Schwarz inequality, the trace inequality \eqref{tracein}, the estimate \eqref{error1} with $m=k$ and $s=0, 1$, and \eqref{normeq}, we have 
\begin{equation}\label{erroreqn2}
\begin{split}
&\Big|\sum_{T\in {\cal T}_h}\sum_{i,j=1}^2 \langle v_{gi}-\partial_i v_0, (Q_r-I) \partial_{ij}^2 u n_j\rangle_{\partial T}\Big|\\
\leq & C(\sum_{T\in {\cal T}_h}\sum_{i=1}^2h_T^{-1}\| v_{gi}-\partial_i v_0\|^2_{\partial T} )^{\frac{1}{2}}  (\sum_{T\in {\cal T}_h} \sum_{i,j=1}^2h_T\|(Q_r-I) \partial_{ij}^2 u n_j\|^2_{\partial T})^{\frac{1}{2}}\\
\leq & C \| v\|_{2,h} (\sum_{T\in {\cal T}_h} \sum_{i,j=1}^2\|(Q_r-I) \partial_{ij}^2 u n_j\|^2_{ T}+h_T^2\|(Q_r-I) \partial_{ij}^2 u n_j\|^2_{1, T})^{\frac{1}{2}}\\
\leq & C  \| v\|_{2,h} h^{k-1}\|u\|_{k+1}\\
\leq & Ch^{k-1}\|u\|_{k+1} \3bar v\3bar.
\end{split}
\end{equation}

Substituting \eqref{erroreqn1}-\eqref{erroreqn2}  into \eqref{erroreqn}  gives
\begin{equation}\label{err}
(\partial^2_{ij, w} e_h, \partial^2_{ij, w}  v)\leq   Ch^{k-1} \|u\|_{k+1} \3bar  v\3bar.
\end{equation}

Using Cauchy-Schwarz inequality, letting $v=Q_hu-u_h$ in \eqref{err}, the error estimate \eqref{erroresti1} gives
\begin{equation*}
\begin{split}
& \3bar u-u_h\3bar^2\\=&\sum_{T\in {\cal T}_h}\sum_{i,j=1}^2(\partial^2_{ij, w}  (u-u_h), \partial^2_{ij, w}  (u-Q_hu))_T+(\partial^2_{ij, w}  (u-u_h), \partial^2_{ij, w}  (Q_hu-u_h))_T\\
\leq &\3bar u-u_h \3bar  \3bar u-Q_hu \3bar+ Ch^{k-1} \|u\|_{k+1}  \3bar Q_hu-u_h\3bar \\
\leq &\3bar u-u_h  \3bar  \3bar u-Q_hu \3bar + Ch^{k-1} \|u\|_{k+1}    (\3bar Q_hu-u\3bar+\3bar u-u_h\3bar)  \\
\leq &\3bar u-u_h  \3bar  \3bar u-Q_hu \3bar + Ch^{k-1}\|u\|_{k+1}   h^{k-1} \|u\|_{k+1}  +Ch^{k-1} \|u\|_{k+1}   \3bar u-u_h\3bar,
\end{split}
\end{equation*}
which further gives
\begin{equation*}
\begin{split}
 \3bar u-u_h\3bar \leq  \3bar u-Q_hu \3bar+Ch^{k-1} \|u\|_{k+1} \leq Ch^{k-1} \|u\|_{k+1}.
\end{split}
\end{equation*} 

This completes the proof.
\end{proof}

\section{  $L^2$ Error Estimates}
To derive the error estimate in the $L^2$ norm, we use the standard duality argument. The error is expressed as $e_h=u-u_h=\{e_0, e_b, \be_g\}$, and we define $\zeta_h =Q_hu - u_h=\{\zeta_0, \zeta_b, \bzeta_g\}\in V_h^0$. Consider the dual problem associated with the biharmonic equation \eqref{model}, which seeks a function $w \in H_0^2(\Omega)$ satisfying:
\begin{equation}\label{dual}
\begin{split}
    \Delta^2 w&=\zeta_0, \qquad \text{in}\ \Omega,\\
w&=0,     \qquad \text{on}\ \partial\Omega,\\
\frac{\partial w}{\partial \bn}&=0,  \qquad \text{on}\ \partial\Omega.
    \end{split}
\end{equation}
We assume the following regularity condition for the dual problem:
\begin{equation}\label{regu2}
 \|w\|_4\leq C\|\zeta_0\|.
 \end{equation}
 
 \begin{theorem}
Let $u\in H^{k+1}(\Omega)$ be the exact solution of the biharmonic equation \eqref{model}, and let 
 $u_h\in V_h$ denote the numerical solution  obtained  using  the weak Galerkin scheme \ref{PDWG1}. Assume that the $H^4$-regularity condition   \eqref{regu2}   holds. Then, there exists a constant $C$ such that 
\begin{equation*}
\|e_0\|\leq Ch^{k+1}\|u\|_{k+1}.
\end{equation*}
 \end{theorem}
 
 \begin{proof}
 Testing the dual problem   \eqref{dual} with  $\zeta_0$ and applying  integration by parts,  we  derive:
 \begin{equation}\label{e1}
 \begin{split}
 \|\zeta_0\|^2 =&(\Delta^2 w, \zeta_0)\\ 
  =& \sum_{T\in {\cal T}_h}\sum_{i, j=1}^2(\partial^2_{ij} w, \partial^2_{ij}\zeta_0)_T-\langle \partial^2_{ij} w, \partial_i\zeta_0 \cdot n_j \rangle_{\partial T}+\langle  \partial_j(\partial^2_{ij} w)\cdot n_i,  \zeta_0  \rangle_{\partial T}\\
    =& \sum_{T\in {\cal T}_h}\sum_{i, j=1}^2(\partial^2_{ij} w, \partial^2_{ij}\zeta_0)_T-\langle \partial^2_{ij} w, (\partial_i\zeta_0-\zeta_{gi}) \cdot n_j \rangle_{\partial T}\\&+\langle  \partial_j(\partial^2_{ij} w)\cdot n_i,  \zeta_0-\zeta_b  \rangle_{\partial T},
 \end{split}
 \end{equation}
 where we used  $\sum_{T\in {\cal T}_h} \sum_{i, j=1}^2 \langle \partial^2_{ij} w,  \zeta_{gi}  \cdot n_j \rangle_{\partial T}=\sum_{i, j=1}^2\langle \partial^2_{ij} w,  \zeta_{gi}  \cdot n_j \rangle_{\partial \Omega}=0$ due to $\bzeta_g =0$ on $\partial\Omega$, and $\sum_{T\in {\cal T}_h} \sum_{i, j=1}^2 \langle  \partial_j(\partial^2_{ij} w)\cdot n_i,  \zeta_b  \rangle_{\partial T} =\sum_{i, j=1}^2\langle  \partial_j(\partial^2_{ij} w)\cdot n_i,  \zeta_b  \rangle_{\partial \Omega}=0$ due to $\zeta_b=0$ on $\partial\Omega$.

 Letting  $u=w$ and $v=\zeta_h$ in \eqref{54} gives
\begin{equation*}
    \begin{split}
     &  \sum_{T\in {\cal T}_h}\sum_{i,j=1}^2(\partial_{ij, w}^2 w, \partial_{ij, w}^2 \zeta_h)_T \\ = &\sum_{T\in {\cal T}_h}\sum_{i,j=1}^2(\partial_{ij}^2 w, \partial_{ij}^2 \zeta_0)_T -
  \langle (\zeta_b-\zeta_0) n_i, \partial_j (Q_r \partial_{ij}^2 w) \rangle_{\partial T}+\langle \zeta_{gi}-\partial_i \zeta_0, Q_r \partial_{ij}^2 w n_j\rangle_{\partial T}, 
    \end{split}
\end{equation*}
   which is equivalent to 
\begin{equation*}
    \begin{split}
& \sum_{T\in {\cal T}_h}\sum_{i,j=1}^2(\partial_{ij}^2 w, \partial_{ij}^2 \zeta_0)_T \\=&\sum_{T\in {\cal T}_h}\sum_{i,j=1}^2(\partial_{ij, w}^2 w, \partial_{ij, w}^2 \zeta_h)_T +\langle (\zeta_b-\zeta_0) n_i, \partial_j (Q_r \partial_{ij}^2 w) \rangle_{\partial T}-\langle \zeta_{gi}-\partial_i \zeta_0, Q_r \partial_{ij}^2 w n_j\rangle_{\partial T}.
  \end{split}
\end{equation*}
Substituting the above equation into \eqref{e1} and using \eqref{erroreqn} gives
\begin{equation}\label{e2}
 \begin{split}
  \|\zeta_0\|^2   
  = & \sum_{T\in {\cal T}_h}\sum_{i,j=1}^2(\partial_{ij, w}^2 w, \partial_{ij, w}^2 \zeta_h)_T +\langle (\zeta_b-\zeta_0) n_i, \partial_j ((Q_r-I) \partial_{ij}^2 w) \rangle_{\partial T}\\&-\langle \zeta_{gi} -\partial_i \zeta_0, (Q_r-I) \partial_{ij}^2 w n_j\rangle_{\partial T}\\
  =& \sum_{T\in {\cal T}_h}\sum_{i,j=1}^2(\partial_{ij, w}^2 w, \partial_{ij, w}^2 e_h)_T+(\partial_{ij, w}^2 w, \partial_{ij, w}^2 (Q_hu-u))_T-\ell(w, \zeta_h)\\ 
  =& \sum_{T\in {\cal T}_h}\sum_{i,j=1}^2(\partial_{ij, w}^2 Q_hw, \partial_{ij, w}^2 e_h)_T+(\partial_{ij, w}^2 (w-Q_hw), \partial_{ij, w}^2 e_h)_T\\&+(\partial_{ij, w}^2 w, \partial_{ij, w}^2 (Q_hu-u))_T-\ell(w, \zeta_h)\\ 
  =&\ell(u, Q_hw) + \sum_{T\in {\cal T}_h}\sum_{i,j=1}^2(\partial_{ij, w}^2 (w-Q_hw), \partial_{ij, w}^2 e_h)_T\\&+(\partial_{ij, w}^2w, \partial_{ij, w}^2(Q_hu-u))_T-\ell(w, \zeta_h)\\
  =&J_1+J_2+J_3+J_4.
 \end{split}
 \end{equation}
 
We will estimate the four terms $J_i(i=1,\cdots,4)$ on the last line of \eqref{e2}  individually.

For $J_1$, using Cauchy-Schwarz inequality, the trace inequality \eqref{tracein}, the inverse inequality,  the estimate \eqref{error1} with $m=k$ and $s=0,1, 2$, the estimate \eqref{error2} with $m=3$ and $s=0,1, 2$, gives
\begin{equation}\label{ee1}
\begin{split}
&J_1= \ell(u, Q_hw)\\
\leq &\Big|\sum_{T\in {\cal T}_h}\sum_{i,j=1}^2 -  \langle (Q_bw-Q_0w) n_i, \partial_j ((Q_r-I) \partial_{ij}^2 u) \rangle_{\partial T}\\&+\langle  Q_n(\partial_i w) -\partial_i Q_0w, (Q_r-I) \partial_{ij}^2 u n_j\rangle_{\partial T}\Big|\\
\leq& \Big(\sum_{T\in {\cal T}_h}\sum_{i=1}^2\|(Q_bw-Q_0w) n_i\|_{\partial T}^2\Big)^{\frac{1}{2}} \Big(\sum_{T\in {\cal T}_h}\sum_{i,j=1}^2\|\partial_j ((Q_r-I) \partial_{ij}^2 u)\|_{\partial T}^2\Big)^{\frac{1}{2}} \\
&+\Big(\sum_{T\in {\cal T}_h}\sum_{i=1}^2\|Q_n(\partial_i w) -\partial_i Q_0w\|_{\partial T}^2\Big)^{\frac{1}{2}} \Big(\sum_{T\in {\cal T}_h}\sum_{i,j=1}^2\|(Q_r-I) \partial_{ij}^2 u n_j\|_{\partial T}^2\Big)^{\frac{1}{2}} \\
\leq& \Big(\sum_{T\in {\cal T}_h} h_T^{-1}\| w-Q_0w \|_{  T}^2+h_T \|w-Q_0w \|_{1, T}^2\Big)^{\frac{1}{2}} \\&\cdot\Big(\sum_{T\in {\cal T}_h}\sum_{i,j=1}^2h_T^{-1}\|\partial_j ((Q_r-I) \partial_{ij}^2 u)\|_{T}^2+h_T\|\partial_j ((Q_r-I) \partial_{ij}^2 u)\|_{1, T}^2\Big)^{\frac{1}{2}} \\
&+\Big(\sum_{T\in {\cal T}_h}\sum_{i=1}^2h_T^{-1}\| \partial_i w  -\partial_i Q_0w\|_{T}^2+h_T \| \partial_i w -\partial_i Q_0w\|_{1, T}^2\Big)^{\frac{1}{2}} \\&\cdot \Big(\sum_{T\in {\cal T}_h}\sum_{i,j=1}^2h_T^{-1}\|(Q_r-I) \partial_{ij}^2 u n_j\|_{T}^2+h_T\|(Q_r-I) \partial_{ij}^2 u n_j\|_{1, T}^2\Big)^{\frac{1}{2}} \\ 
\leq & Ch^{k+1}\|u\|_{k+1}\|w\|_4.
\end{split}
\end{equation}

For $J_2$, using Cauchy-Schwarz inequality, \eqref{erroresti1} with $k=3$ and \eqref{trinorm} gives
\begin{equation}\label{ee2}
\begin{split}
J_2\leq \3bar w-Q_hw\3bar \3bar e_h\3bar\leq Ch^{k-1}\|u\|_{k+1}h^2\|w\|_4\leq Ch^{k+1}\|u\|_{k+1}\|w\|_4.
\end{split}
\end{equation}

  For $J_3$, denote by $Q^1$ a $L^2$ projection onto $P_1(T)$. Using \eqref{2.4} gives
  \begin{equation}\label{ee}
 \begin{split}
  &(\partial^2_{ij, w}(Q_hu-u), Q^1\partial^2_{ij, w} w)_T\\
  =& (Q_0u-u, \partial^2_{ji} ( Q^1\partial^2_{ij, w} w))_T-\langle Q_bu-u, \partial_j (Q^1\partial^2_{ij, w} w)\rangle_{\partial T}\\&+ \langle Q_n(\partial_i u)-\partial_i u, Q^1\partial^2_{ij, w} w n_j\rangle_{\partial T}=0,
 \end{split}
 \end{equation}
 where we used $ \partial^2_{ji} ( Q^1\partial^2_{ij, w} w)=0$,  $ \partial_j (Q^1\partial^2_{ij, w} w)=C$ and the property of the projection  operators $Q_b$ and $Q_n$ and $p\geq q\geq 1$.

  Using \eqref{ee}, Cauchy-Schwarz inequality, \eqref{pro} and \eqref{erroresti1}, gives 
  \begin{equation}\label{ee3}
  \begin{split}
  J_3\leq &|\sum_{T\in {\cal T}_h}\sum_{i, j=1}^2(\partial^2_{ij, w} w, \partial^2_{ij, w} (Q_hu-u))_T|
  \\
  =&|\sum_{T\in {\cal T}_h}\sum_{i, j=1}^2(\partial^2_{ij, w} w-Q^1\partial^2_{ij, w} w, \partial^2_{ij, w} (Q_hu-u))_T|\\
  =&|\sum_{T\in {\cal T}_h}\sum_{i, j=1}^2(Q_r\partial^2_{ij}   w-Q^1 Q_r\partial^2_{ij} w, \partial^2_{ij, w} (Q_hu-u))_T|\\
  \leq & \Big(\sum_{T\in {\cal T}_h}\sum_{i, j=1}^2\|Q_r\partial^2_{ij}   w-Q^1 Q^r\partial^2_{ij}  w\|_T^2\Big)^{\frac{1}{2}} \3bar Q_hu-u \3bar\\
  \leq & Ch^{k+1}\|u\|_{k+1} \|w\|_4.
  \end{split}
  \end{equation}

For $J_4$, using Cauchy-Schwarz inequality, the trace inequality \eqref{tracein}, Lemma \ref{normeqva},   the estimate \eqref{error1} with $m=3$ and $s=0, 1$,  \eqref{erroresti1}, \eqref{trinorm} gives
\begin{equation}\label{ee4}
\begin{split}
J_4=&\ell(w, \zeta_h)\\
\leq &\Big|\sum_{T\in {\cal T}_h}\sum_{i,j=1}^2 -  \langle (\zeta_b-\zeta_0) n_i, \partial_j ((Q_r-I) \partial_{ij}^2 w) \rangle_{\partial T}\\&+\langle \zeta_{gi}-\partial_i \zeta_0, (Q_r-I) \partial_{ij}^2 w n_j\rangle_{\partial T}\Big| \\
\leq& \Big(\sum_{T\in {\cal T}_h}\sum_{i=1}^2\|(\zeta_b-\zeta_0) n_i\|_{\partial T}^2\Big)^{\frac{1}{2}} \Big(\sum_{T\in {\cal T}_h}\sum_{i,j=1}^2\|\partial_j ((Q_r-I) \partial_{ij}^2 w) \|_{\partial T}^2\Big)^{\frac{1}{2}}   \\
&+\Big(\sum_{T\in {\cal T}_h}\sum_{i=1}^2\|\zeta_{gi}-\partial_i \zeta_0\|_{\partial T}^2\Big)^{\frac{1}{2}} \Big(\sum_{T\in {\cal T}_h}\sum_{i,j=1}^2\|(Q_r-I) \partial_{ij}^2 w n_j\|_{\partial T}^2\Big)^{\frac{1}{2}}  \\
\leq& \Big(\sum_{T\in {\cal T}_h}\sum_{i,j=1}^2h_T^2 \|\partial_j ((Q_r-I) \partial_{ij}^2 w)\|_{T}^2+h_T^4 \| \partial_j ((Q_r-I) \partial_{ij}^2 w)\|_{1, T}^2\Big)^{\frac{1}{2}} \\&\cdot\Big(\sum_{T\in {\cal T}_h}h_T^{-3}\|\zeta_0-\zeta_b\|_{\partial T}^2\Big)^{\frac{1}{2}}  \\
 &+ \Big(\sum_{T\in {\cal T}_h} \sum_{i,j=1}^2 \|(Q_r-I) \partial_{ij}^2 w n_j\|_{T}^2+h_T^2 \|(Q_r-I) \partial_{ij}^2 w n_j\|_{1, T}^2\Big)^{\frac{1}{2}} \\ &\cdot\Big(\sum_{T\in {\cal T}_h}\sum_{i=1}^2 h_T^{-1}\|\zeta_{gi}-\partial_i \zeta_0\|_{\partial T}^2\Big)^{\frac{1}{2}}  \\\leq& Ch^2\|w\|_{4}\3bar\zeta_h\3bar   \\
\leq &Ch^2\|w\|_{4}(\3bar u-u_h\3bar+\3bar u-Q_hu\3bar)  \\
\leq &Ch^{k+1}\|w\|_4\|u\|_{k+1}.
\end{split}
\end{equation}

 Substituting \eqref{ee1}-\eqref{ee2} and \eqref{ee3}-\eqref{ee4} into \eqref{e2}, and using \eqref{regu2},   gives
$$
\|\zeta_0\|^2\leq Ch^{k+1}\|w\|_4\|u\|_{k+1}\leq Ch^{k+1}  \|u\|_{k+1} \|\zeta_0\|.
$$
This gives
$$
\|\zeta_0\|\leq Ch^{k+1} \|u\|_{k+1},
$$
which, using the triangle inequality and \eqref{error2} with $m=k$ and $s=0$, gives
$$
\|e_0\|\leq \|\zeta_0\|+\|u-Q_0u\|\leq Ch^{k+1}\|u\|_{k+1}. 
$$

This completes the proof of the theorem. 
\end{proof}

\end{document}